\documentclass{article}

\title{Left adjoint to precomposition in elementary doctrines}
\author{Francesca Guffanti\footnote{The author would like to acknowledge her Ph.D. supervisors: Sandra Mantovani and Pino Rosolini. This paper is the fourth chapter of her Ph.D. thesis, defended at Universit`a degli Studi di Milano. The author is also thankful to the anonymous referee for the suggestions that led to this improved version of the manuscript.}}
\date{}

\usepackage[utf8]{inputenc}
\usepackage{csquotes}
\usepackage[italian,british]{babel}
\usepackage[margin=89.5pt]{geometry}
\usepackage{hyperref}

\hypersetup{
    pdftitle={},
    pdfauthor={},
    pdfmenubar=false,
    pdffitwindow=true,
    pdfstartview=FitH,
    colorlinks=true,
    linkcolor=black,
    citecolor=black,
    urlcolor=cyan
}

\usepackage{todonotes}
\usepackage{setspace}
\usepackage{mathtools}
\usepackage{amssymb}
\usepackage{bbold}
\usepackage{amsmath,amsthm, graphicx,xspace}
\usepackage{mathrsfs}
\usepackage{latexsym}
\usepackage{rotating}
\usepackage{multicol}
\usepackage{verbatim}
\usepackage{faktor}
\usepackage{epigraph}
\usepackage{quiver}
\usepackage{esvect}
\usepackage{cancel}
\usepackage{oubraces}
\usepackage{bbm}

\usepackage{enumitem}
\setlist[itemize]{parsep=0pt}
\setlist[enumerate]{parsep=0pt}
\usepackage{tikz-cd}
\usepackage{tikz}
\usetikzlibrary{arrows}

\numberwithin{equation}{section}
\theoremstyle{definition}
\newtheorem{theorem}{Theorem}[section]
\newtheorem{proposition}[theorem]{Proposition}
\newtheorem{lemma}[theorem]{Lemma}
\newtheorem{claim}[theorem]{Claim}

\newtheorem{definition}[theorem]{Definition}

\newtheorem{example}[theorem]{Example}

\mathcode`\:="603A 
\mathchardef\colon="303A 

\def\pws{\mathop{\mathscr{P}\kern-1.1ex_{\ast}}}
\DeclareMathOperator{\OPRid}{id}
\newcommand{\id}[1]{\ensuremath{\OPRid_{#1}}}
\newcommand{\ct}[1]{\ensuremath{\mathbb{#1}}\xspace}
\newcommand{\CC}{\ct{C}}
\newcommand{\Pos}{\mathbf{Pos}}
\newcommand{\ple}[1]{\langle#1\rangle}
\newcommand{\op}{^{\mathrm{op}}}
\newcommand{\blank}{\mathrm{-}}
\newcommand{\pr}[1]{\ensuremath{\mathrm{pr}_{#1}}}
\newcommand{\tmn}{\mathbf{t}}
\newcommand{\ED}{\mathbf{ED}}
\newcommand{\QED}{\mathbf{QED}}
\newcommand{\Sub}{\mathop{\mathrm{Sub}}}
\newcommand{\Lan}{\mathop{\mathrm{Lan}}}
\newcommand{\Set}{\textnormal{Set}}

\newcommand{\des}{\mathscr{D}es}
\newcommand{\dom}{\mathop{\mathrm{dom}}}
\newcommand{\Mod}{\mathop{\mathrm{Mod}}}
\newcommand{\fbf}{\mathtt{LT}}
\newcommand{\efbf}{\mathtt{HF}}
\newcommand{\ctx}{\mathrm{\mathbb{C}tx}}
\newcommand{\pointinproof}[1]{\newline{\bf#1:}}
\begin{document}
\newlength{\myindent} 
\setlength{\myindent}{\parindent}
\parindent 0em 

\maketitle
\begin{abstract}
It is well-known in universal algebra that adding structure and equational axioms generates forgetful functors between varieties, and such functors all have left adjoints. The category of elementary doctrines provides a natural framework for studying algebraic theories, since each algebraic theory can be described by some syntactic doctrine and its models are homomorphism from the syntactic doctrine into the doctrine of subsets. In this context, adding structure and axioms to a theory can be described by a homomorphism between the two corresponding syntactic doctrines, and the forgetful functor arises as precomposition with this last homomorphism. In this work, given any homomorphism of elementary doctrines, we prove the existence of a left adjoint of the functor induced by precomposition in the doctrine of subobjects of a Grothendieck topos.
\end{abstract}
\section{Introduction} 
In universal algebra, adding structure or equational axioms is a widely used technique: classical results say that for a given category of algebraic structure---e.g.\ monoids---, adding some structure or axioms---e.g.\ groups, commutative monoids---defines a forgetful functor from the new category to the original one, with a left adjoint. The 2-category $\ED$ of elementary doctrines provides a natural framework for studying algebraic theories, with each theory $\mathbb{T}$ for a particular algebraic language ${\Sigma}$ described by some doctrine of formulae $\efbf^{\Sigma}_\mathbb{T}$.
From a categorical point of view, every variety is equivalent to a category of homomorphisms of elementary doctrines $\ED(\efbf^{\Sigma}_\mathbb{T},\pws)$ between a doctrine of formulae and the subsets doctrine $\pws :\Set_{\ast}\op\to\Pos$. Moreover, adding structure and equational axioms translates to a doctrine homomorphism $(E,\mathfrak e):\efbf^{\Sigma}_\mathbb{T}\to\efbf^{\Sigma'}_\mathbb{T'}$. Precomposition with this homomorphism induces a functor $\blank\circ(E,\mathfrak e):\ED(\efbf^{\Sigma'}_\mathbb{T'},\pws)\to\ED(\efbf^{\Sigma}_\mathbb{T},\pws)$, and it represents the forgetful functor between the correspondent varieties, hence it has a left adjoint.
\[\begin{tikzcd}
	{\Mod^{\Sigma'}_{\mathbb T'}} & {\Mod^{\Sigma}_{\mathbb T}} \\
	{\ED(\efbf^{\Sigma'}_{\mathbb{T'}},\pws )} & {\ED(\efbf^{\Sigma}_{\mathbb{T}},\pws )}
	\arrow["\rotatebox{-90}{$\cong$}"{description}, draw=none, from=2-1, to=1-1]
	\arrow[from=1-1, to=1-2]
	\arrow["{\blank\circ(E,\mathfrak{e})}", from=2-1, to=2-2]
	\arrow["\rotatebox{-90}{$\cong$}"{description}, draw=none, from=2-2, to=1-2]
\end{tikzcd}\]
In this work we extend this classical result in $\ED$ by considering the subobject doctrine $\Sub:\mathbb{E}\op\to\Pos$ from a Grothendieck topos $\mathbb{E}$ instead of the doctrine of subsets, and precomposition with any homomorphism $(F,\mathfrak f):P\to R$ instead of the forgetful functor. The whole article is dedicated to the proof that also in this case the functor
\begin{equation*}\blank\circ(F,\mathfrak f):\ED(R,\Sub)\to\ED(P,\Sub)\end{equation*}
has a left adjoint, showing how the existence of free functors in universal algebra follows from a more general result that lives in the theory of elementary doctrines.
\section{Elementary doctrines}
We start by recalling the language of elementary doctrines introduced in \cite{quotcomplfoun,unifying,triposes}, as a generalization of Lawvere's hyperdoctrine \cite{adjfound,lawdiag,lawequality}. While in general, doctrines are a way to generalize the posets of well-formed formulae ordered by provable consequence, the particular context of elementary doctrines is a suitable one in which we can interpret conjunctions and equality of formulae. We define the 2-category $\ED$ of elementary doctrines, which will be the main protagonists of this work, and show some relevant examples.
\nocite{borhandbk,elephant,CatWorMat}
\begin{definition}
Let $\ct{C}$ be a category with finite products and let $\Pos$ be the category of partially-ordered sets and monotone functions. A \emph{primary doctrine} is a functor $P:\ct{C}\op\to\Pos$ such that for each object $A$ in $\CC$, the poset $P(A)$ has finite meets, and the related operations $\land:P\times P\xrightarrow{\cdot} P$ and $\top:\mathbf{1}\xrightarrow{\cdot} P$ yield natural transformations. The category $\CC$ is called \emph{base category of $P$}, each poset $P(X)$ for an object $X\in\CC$ is called \emph{fiber}, the function $P(f)$ for an arrow $f$ in $\CC$ is called \emph{reindexing}.
\end{definition}
We recall the definition of elementary doctrines, which are---informally speaking---doctrines in which we can interpret equality. The original definition of equality for hyperdoctrines was given by Lawvere in \cite{lawequality}: for any object $X$ in the base category the equality is defined as $\Sigma_{\Delta_X}(\top_X)$, where $\Sigma_{\Delta_X}:P(X)\to P(X\times X)$ is the left adjoint of the reindexing $P(\Delta_X):P(X\times X)\to P(X)$ of the usual diagonal map $\Delta_X:X\to X\times X$. Since Lawvere's definition of doctrines with equality, there have been many equivalent definitions up to date---some can be found in \cite{ElemQuotCompl,quotcomplfoun}. The one we chose to deal with in this work is taken from the characterization in Proposition 2.5 of \cite{EmPaRo}.
\begin{definition}
A primary doctrine $P:\ct{C}\op\to\Pos$ is \emph{elementary} if for any object $A$ in $\ct{C}$ there exists an element $\delta_A\in P(A\times A)$ such that:
\begin{enumerate}
\item $\top_A\leq P(\Delta_A)(\delta_A)$;
\item $P(A)=\des_{\delta_A}:=\{\alpha\in P(A)\mid P(\pr1)(\alpha)\land\delta_A\leq P(\pr2)(\alpha)\}$;
\item $\delta_A\boxtimes\delta_B\leq\delta_{A\times B}$, where $\delta_A\boxtimes\delta_B=P(\ple{\pr1,\pr3})(\delta_A)\land P(\ple{\pr2,\pr4})(\delta_B)$.
\end{enumerate}
In 2., $\pr1$ and $\pr2$ are the projections from $A\times A$ in $A$; in 3., the projections are from $A\times B\times A\times B$. The element $\delta_A$ will be called \emph{fibered equality} on $A$.
\end{definition}
\begin{example}\label{ex:doctr}
\begin{enumerate}
\item The functor $\mathscr{P}:\Set\op\to\Pos$, sending each set in the poset of its subsets, ordered by inclusion, and each function $f:A\to B$ to the inverse image $f^{-1}:\mathscr{P}(B)\to\mathscr{P}(A)$ is an elementary doctrine. For any set $A$, intersection of two subsets is their meet, $A$ is the top element, the 
subset $\{(a,a)\mid a \in A\}\subseteq A\times A$ is the fibered equality on $A$.
\item\label{item:ex:doctr-2} For a given category $\CC$ with finite limits, the functor $\Sub_\CC:\CC\op\to\Pos$ sending each object to the poset of its subobjects in $\CC$ and each arrow $f:A\to B$ to the pullback function $f^*:\Sub_\CC(B)\to\Sub_\CC(A)$, is an elementary doctrine.
For any object $A$ in $\CC$, the pullback of a subobject along another defines their meet.
\[\begin{tikzcd}
	{\dom(\alpha\land\beta)} & \dom\alpha \\
	\dom\beta & A
	\arrow["{\pi_1}"', tail, from=1-1, to=2-1]
	\arrow["{\pi_2}", tail, from=1-1, to=1-2]
	\arrow["\alpha", tail, from=1-2, to=2-2]
	\arrow["\beta"', tail, from=2-1, to=2-2]
	\arrow["\alpha\land\beta"{description}, tail, from=1-1, to=2-2]
\end{tikzcd}\]
The arrow $\id{A}$ is the top element. The usual diagonal map $\Delta_A:A\rightarrowtail A\times A$ is the fibered equality on $A$---see in \cite{ElemQuotCompl} the Example 2.4.a.
\item For a given theory $\mathcal{T}$ on a one-sorted first-order language $\mathcal{L}$ with equality, let $\ctx_\mathcal{L}$ be the category of contexts: an object is a finite list of distinct variables, and an arrow between two lists $\vec x=(x_1,\dots, x_n)$ and $\vec y=(y_1,\dots, y_m)$ is
\begin{equation*}(t_1(\vec x),\dots,t_m(\vec x)):(x_1,\dots, x_n) \to(y_1,\dots, y_m)\end{equation*}
an $m$-tuple of terms in the context $\vec x$. The functor $\fbf^\mathcal{L}_{\mathcal{T}}:\ctx_\mathcal{L}\op\to\Pos$ sends each list of variables to the poset reflection of well-formed formulae written with at most those variables ordered by provable consequence in $\mathcal{T}$; moreover, $\fbf^\mathcal{L}_{\mathcal{T}}:\ctx_\mathcal{L}\op\to\Pos$ sends an arrow $\vec{t}(\vec{x}):\vec{x}\to\vec{y}$ into the substitution $[\vec{t}(\vec{x})/\vec{y}]$. For any list $\vec x$, the conjunction of two formulae is their binary meet, the true constant $\top$ is the top element, the formula $\big(x_1=x_1'\land\dots\land x_n=x_n'\big)$ in $\fbf^\mathcal{L}_\mathcal{T}(\vec x;\vec x')$ is the fibered equality on $\vec{x}$. The functor $\fbf^\mathcal{L}_{\mathcal{T}}$ is an elementary doctrine.
\item For a given category $\ct{D}$ with finite products and weak pullbacks, the functor of weak subobjects $\mathbf{\Psi}_\ct{D}:\ct{D}\op\to\Pos$ sending each object $A$ to the poset reflection of the comma category $\ct{D}/A$ is an elementary doctrine: for each arrow $f:A\to B$, $\mathbf{\Psi}_\ct{D}(f)$ sends the equivalence class of an arrow $\alpha:\dom\alpha\to B$ to the equivalence class of the projection $\pi_1$ of a chosen weak pullback of $\alpha$ along $f$---see Example 2.9 in \cite{quotcomplfoun} for more details.
\[\begin{tikzcd}
	W & \dom\alpha \\
	A & B
	\arrow["{\pi_1}"', from=1-1, to=2-1]
	\arrow["{\pi_2}", from=1-1, to=1-2]
	\arrow["\alpha", from=1-2, to=2-2]
	\arrow["f"', from=2-1, to=2-2]
\end{tikzcd}\]
For any object $A$ in $\ct{D}$, a choice of a weak pullback of a representative of a weak subobject along another defines their meet.
\[\begin{tikzcd}
	{\dom(\alpha\land\beta)} & \dom\alpha \\
	\dom\beta & A
	\arrow["{\pi_1}"', from=1-1, to=2-1]
	\arrow["{\pi_2}", from=1-1, to=1-2]
	\arrow["\alpha", from=1-2, to=2-2]
	\arrow["\beta"', from=2-1, to=2-2]
	\arrow["\alpha\land\beta"{description}, from=1-1, to=2-2]
\end{tikzcd}\]
The class of $\id{A}$ is the top element. The equivalence class of the usual diagonal map $\Delta_A:A\rightarrowtail A\times A$ is the fibered equality on $A$.
\end{enumerate}
\end{example}
\begin{definition}
An \emph{elementary doctrine homomorphism}---1-cell or 1-arrow---between two elementary doctrines $P:\ct{C}\op\to\Pos$ and $R:\ct{D}\op\to\Pos$ is a pair $(F,\mathfrak f)$ where $F:\ct{C}\to\ct{D}$ is a functor that preserves finite products and $\mathfrak{f}:P\xrightarrow{\cdot}R\circ F\op$ is a natural transformation preserving finite meets and the fibered equality, i.e.\
\begin{equation*}\mathfrak{f}_A(\alpha\land_A\alpha')=\mathfrak{f}_A(\alpha)\land_{FA}\mathfrak{f}(\alpha');\qquad\qquad\mathfrak{f}_A(\top_A)=\top_{FA};\qquad\qquad\mathfrak{f}_{A\times A}(\delta_A)=\delta_{FA}.\end{equation*}
Sometimes a homomorphism between $P$ and $R$ will be called a model of $P$ in $R$.
A \emph{2-cell} between $(F,\mathfrak f)$ and $(G,\mathfrak g)$ from $P$ to $R$ is a natural transformation $\theta:F\xrightarrow{\cdot} G$ such that $\mathfrak{f}_A(\alpha)\leq R(\theta_A)(\mathfrak{g}_A(\alpha))$ for any object $A$ in $\ct{C}$ and $\alpha\in P(A)$. Elementary doctrines, elementary doctrine homomorphisms with 2-cells defined here form a $2$-category, that will be denoted $\ED$.
\end{definition}
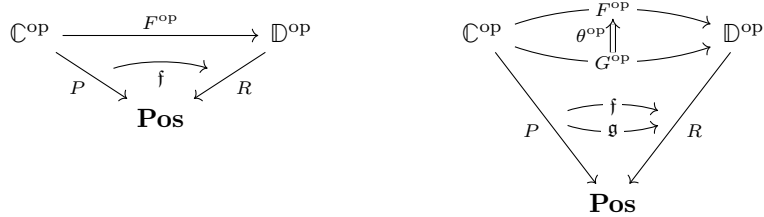
\[\begin{tikzcd}
	{\ct{C}\op} && {\ct{D}\op} && {\ct{C}\op} && {\ct{D}\op} \\
	& \Pos \\
	&&&&& \Pos
	\arrow["F\op", from=1-1, to=1-3]
	\arrow[""{name=0, anchor=center, inner sep=0}, "P"', from=1-1, to=2-2]
	\arrow[""{name=1, anchor=center, inner sep=0}, "R", from=1-3, to=2-2]
	\arrow[""{name=2, anchor=center, inner sep=0}, "F\op"{description}, curve={height=-12pt}, from=1-5, to=1-7]
	\arrow[""{name=3, anchor=center, inner sep=0}, "P"', from=1-5, to=3-6]
	\arrow[""{name=4, anchor=center, inner sep=0}, "R", from=1-7, to=3-6]
	\arrow[""{name=5, anchor=center, inner sep=0}, "G\op"{description}, curve={height=12pt}, from=1-5, to=1-7]
	\arrow["{\mathfrak{f}}"', curve={height=-6pt}, shorten <=8pt, shorten >=8pt, from=0, to=1]
	\arrow["{\mathfrak f}"{description}, curve={height=-6pt}, shorten <=8pt, shorten >=8pt, from=3, to=4]
	\arrow["{\mathfrak g}"{description}, curve={height=6pt}, shorten <=8pt, shorten >=8pt, from=3, to=4]
	\arrow["\theta\op", shorten <=3pt, shorten >=3pt, Rightarrow, from=5, to=2]
\end{tikzcd}\]
\begin{example}
For a given category $\CC$ with finite limits, the inclusion of $\Sub_\CC(A)$ into the poset reflection of $\CC/A$ yields a natural transformation $\Sub_\CC\to\mathbf{\Psi}_\CC$; pairing it with the identity on the base category $\CC$, this defines a 1-arrow in $\ED$.
\end{example}
\section{The definition of the left adjoint functor}
Fix in the category $\ED$ of elementary doctrines a homomorphism $(F,\mathfrak{f})$ between two elementary doctrines $P$ and $R$: 
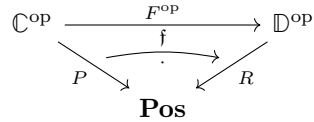
\begin{center}
\begin{tikzcd}
\CC\op\arrow[rr,"F\op"] \arrow[dr,"P"' ,""{name=L}]&&\mathbb{D}\op\arrow[dl,"R" ,""'{name=R}]\\
&\Pos\arrow[rightarrow,"\mathfrak{f}","\cdot"', from=L, to=R, bend left=10]
\end{tikzcd}
\end{center}
where $F:\CC\to\mathbb{D}$ is a product preserving functor, $\mathfrak{f}:P\xrightarrow{\cdot} RF\op$ is a natural transformation that preserves meets, top element and the elementary structure. Moreover, suppose that $\CC$ is small.
Consider a Grothendieck topos $\mathbb{E}$, and the associated subobjects doctrine $\Sub:\mathbb{E}\op\to\Pos$, which is elementary---indeed, it is enough to ask for a finitely complete base category, see Example \ref{ex:doctr}\eqref{item:ex:doctr-2}. Trivially we can precompose any homomorphism $(K,\mathfrak k):R\to\Sub$ in $\ED$ with $(F,\mathfrak f)$ to obtain a homomorphism $(K,\mathfrak k)(F,\mathfrak f):P\to\Sub$; this gives a functor
\begin{equation*}\blank\circ(F,\mathfrak f):\ED(R,\Sub)\to\ED(P,\Sub).\end{equation*}
We look for a left adjoint for this precomposition.
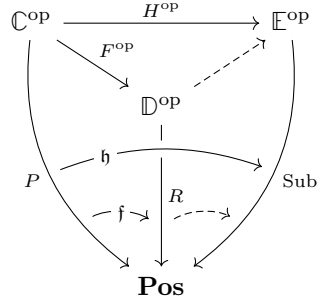
\begin{center}
\begin{tikzcd}
\CC\op\arrow[dr,"F\op"]\arrow[ddddr,"P"'{name=f}, bend right,""'{name=a, near end}]\arrow[rr,"H\op"]&&\mathbb{E}\op\arrow[ddddl,"\Sub"{name=g}, bend left,""{name=e, near end}]\\
&\mathbb{D}\op\arrow[ddd,"R", ""{name=b, near end},""'{name=c, near end}]\arrow[ur,dashed]\\ \\ \\
&\Pos
\arrow[from=a, to=b, bend left, shorten=2mm, "\mathfrak{f}" description]
\arrow[from=c, to=e, bend left, shorten=2mm, dashed]
\arrow[from=f, to=g, bend left=20, shorten=2mm, "\mathfrak{h}"{description, near start}, crossing over]
\end{tikzcd}
\end{center}
The whole section is devoted to the proof of the following:
\begin{theorem}\label{thm:left_adj_prec}
Let $(F,\mathfrak f):P\to R$ be a homomorphism in $\ED$, and suppose the base category of $P$ to be small. Moreover let $\ct{E}$ be a Grothendieck topos and $\Sub:\ct{E}\op\to\Pos$ be the subobject doctrine. Then, the functor induced by precomposition
\begin{equation*}\blank\circ(F,\mathfrak f):\ED(R,\Sub)\to\ED(P,\Sub)\end{equation*}
has a left adjoint.
\end{theorem}
We start from a homomorphism $(H,\mathfrak h):P\to\Sub$, our first goal is to find a functor $\mathbb{D}\to\mathbb{E}$. An easy choice is to take the left Kan extension of $H$ along $F$, whose existence is granted by the fact that $\mathbb{E}$ is a Grothendieck topos---since it is cocomplete, see Chapter X in \cite{CatWorMat}. Recall that the left Kan extension comes with a natural transformation $\mu:H\xrightarrow{\cdot}(\Lan_FH) F$ such that for any other functor $K:\mathbb{D}\to\mathbb{E}$ and any other natural transformation $\theta:H\xrightarrow{\cdot} KF$ there exists a unique $\widehat{\theta}:\Lan_FH\xrightarrow{\cdot} K$ making the obvious diagrams commute:
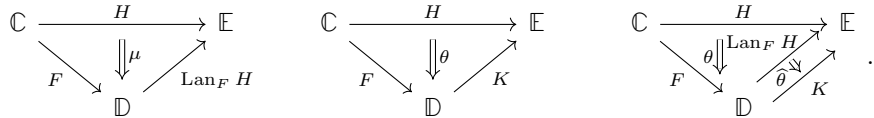
\[\begin{tikzcd}
	\CC && {\mathbb{E}} & \CC && {\mathbb{E}} & \CC && {\mathbb{E}} \\
	& {\mathbb{D}} &&& {\mathbb{D}} &&& {\mathbb{D}}
	\arrow["F"', from=1-1, to=2-2]
	\arrow[""{name=0, anchor=center, inner sep=0}, "H", from=1-1, to=1-3]
	\arrow["F"', from=1-4, to=2-5]
	\arrow[""{name=1, anchor=center, inner sep=0}, "H", from=1-7, to=1-9]
	\arrow["F"', from=1-7, to=2-8]
	\arrow[""{name=2, anchor=center, inner sep=0}, "{\Lan_FH}"{xshift=1.5ex}, shift left=2, from=2-8, to=1-9]
	\arrow[""{name=3, anchor=center, inner sep=0}, "K"', shift right=2, from=2-8, to=1-9]
	\arrow["{\Lan_FH}"', from=2-2, to=1-3]
	\arrow[""{name=4, anchor=center, inner sep=0}, "H", from=1-4, to=1-6]
	\arrow["K"', from=2-5, to=1-6]
	\arrow["\mu", shorten <=5pt, shorten >=3pt, Rightarrow, from=0, to=2-2]
	\arrow["\theta", shorten <=5pt, shorten >=3pt, Rightarrow, from=4, to=2-5]
	\arrow["{\widehat{\theta}}"'{yshift=0.4ex, xshift=-0.3ex}, shorten <=1pt, shorten >=1pt, Rightarrow, from=2, to=3]
	\arrow["\theta"',shift right=9pt, shorten <=5pt, shorten >=5pt, Rightarrow, from=1, to=2-8]
\end{tikzcd}.\]
Before we continue, we need $\Lan_FH$ to be product preserving.
\begin{proposition}
Let $\CC,\mathbb{D},\mathbb{E}$ be categories with finite products such that $\CC$ is small and $\mathbb{E}$ is cocomplete and cartesian closed, and let $F:\CC\to\mathbb{D}$ and $H:\CC\to\mathbb{E}$ be finite product preserving functors. Then $\Lan_FH:\mathbb{D}\to\mathbb{E}$ preserves finite products.\end{proposition}
This is a classical result in the theory of Kan extension that can be shown in different ways, see for instance Proposition 2.5 of \cite{kelly}, or \cite{kan}.
 
If $\mathbb{E}$ is a Grothendieck topos, the hypothesis of the proposition above are satisfied, so $\Lan_FH$ preserves finite products.

Define now a natural transformation $\mathfrak l :R\xrightarrow{\cdot}\Sub(\Lan_FH)\op$. For any object $D\in\mathbb{D}$, and any $\gamma\in R(D)$, write
\begin{equation*}\mathfrak l _D(\gamma)=\bigwedge_{(K,\mathfrak k),\theta}\widehat{\theta}_D^*(\mathfrak k _D(\gamma))\end{equation*}
where $(K,\mathfrak k):R\to\Sub$ is an arrow in $\ED$ and $\theta:(H,\mathfrak h)\to(K,\mathfrak k)(F,\mathfrak f)$ is a 2-arrow, i.e.\ $\mathfrak{h}_A(\alpha)\leq \theta^*_A(\mathfrak{k}_{FA}(\mathfrak{f}_A(\alpha)))$ for all $A\in\CC$ and $\alpha\in P(A)$. Observe that $\mathfrak k _D(\gamma)$ is a subobject of $KD$, $\widehat{\theta}$ is defined by the universal property of the left Kan extension, and $\widehat{\theta}_D^*(\mathfrak k _D(\gamma))$ is the pullback of $\mathfrak k _D(\gamma)$ along $\widehat{\theta}_D:(\Lan_FH)(D)\to KD$, hence it is a subobject of $(\Lan_FH)(D)$. Since $\mathbb{E}$ is a complete category, the infimum of $\{\widehat{\theta}_D^*(\mathfrak k _D(\gamma))\}_{(K,\mathfrak k),\theta}$ exists, and we call it $\mathfrak l _D(\gamma)$.
\begin{lemma} The following properties hold:
\begin{enumerate}
\item $\mathfrak l :R\xrightarrow{\cdot}\Sub(\Lan_FH)\op$ is a natural transformation;
\item $\mathfrak l :R\xrightarrow{\cdot}\Sub(\Lan_FH)\op$ preserves finite meets;
\item $\mathfrak l_{D\times D}(\delta_D)\in\Sub\big((\Lan_FH)(D)\times(\Lan_FH)(D)\big)$ is an equivalence relation for any object $D\in\mathbb{D}$.
\end{enumerate}\end{lemma}
\begin{proof}
\begin{enumerate}
\item Take an arrow $g:D'\to D$ in $\mathbb{D}$, we prove that $\big((\Lan_FH)(g)\big)^*\mathfrak{l}_D(\gamma)=\mathfrak{l}_{D'}R(g)(\gamma)$ for any $\gamma\in RD$:
\begin{align*}\big((\textstyle{\Lan_FH})(g)\big)^*\mathfrak{l}_D(\gamma)&=\big((\textstyle{\Lan_FH})(g)\big)^*\left(\bigwedge_{(K,\mathfrak k),\theta}\widehat{\theta}_D^*(\mathfrak k _D(\gamma))\right)\\
&=\bigwedge_{(K,\mathfrak k),\theta}\big((\textstyle{\Lan_FH})(g)\big)^*\widehat{\theta}_D^*(\mathfrak k _D(\gamma))\\
&=\bigwedge_{(K,\mathfrak k),\theta}\widehat{\theta}_{D'}^*\big(K(g)\big)^*(\mathfrak k _D(\gamma))\\
&=\bigwedge_{(K,\mathfrak k),\theta}\widehat{\theta}_{D'}^*(\mathfrak k _{D'}R(g)(\gamma))=\mathfrak l _{D'}(R(g)\gamma).\end{align*}
Note that the second equality follows from the fact that pullback functors between subobjects categories preserve arbitrary limits---since in a regular categories they have a left adjoint---; the other equalities follow from naturality of $\widehat{\theta}$ and $\mathfrak k$.
\item The top element $\top_D\in RD$ for any object $D\in\mathbb{D}$ is preserved by $\mathfrak l _D$ since $\mathfrak k_D(\top_D)$ is $\id{KD}:KD\to KD$ the top element in $\Sub(KD)$ by assumption, and its pullback along any $\widehat{\theta}_D$ is the identity of $(\Lan_FH)(D)$. Similarly, $\mathfrak l _D$ preserves binary meets since any $\mathfrak k _D$ and any pullback functor do.
\item Compute
\begin{equation*}\mathfrak l_{D\times D}(\delta_D)=\bigwedge_{(K,\mathfrak k),\theta}\widehat{\theta}_{D\times D}^*(\mathfrak k _{D\times D}(\delta_D))=\bigwedge_{(K,\mathfrak k),\theta}{(\widehat{\theta}_{D}\times\widehat{\theta}_{D})}^*(\Delta_{KD}).\end{equation*}
Note that each ${(\widehat{\theta}_{D}\times\widehat{\theta}_{D})}^*(\Delta_{KD})$ is an equivalence relation on $(\Lan_FH)(D)$, since it is the kernel pair of the map $\widehat{\theta}_{D}$. So $\mathfrak{l}_{D\times D}(\delta_D)$ is an equivalence relation itself, being the infimum of equivalence relations.\qedhere
\end{enumerate}
\end{proof}

Before we go on, we recall from Section 4 of \cite{ElemQuotCompl} the \emph{elementary quotient completion} of an elementary doctrine. Given any elementary doctrine $P:\CC\op\to\Pos$ one can build the category $\mathcal{R}_P$ of equivalence relations of $P$: objects are pairs $(A,\rho)$, where $\rho\in P(A\times A)$ is a $P$-equivalence relation on $A$---meaning that it satisfies reflexivity $\delta_A\leq_{A\times A}\rho$, symmetry $\rho\leq_{A\times A}P(\ple{\pr2,\pr1})(\rho)$ and transitivity $P(\ple{\pr1,\pr2})(\rho)\land P(\ple{\pr2,\pr3})(\rho)\leq_{A\times A\times A}P(\ple{\pr1,\pr3})(\rho)$; arrows $ f :(A,\rho)\to(B,\sigma)$ in $\mathcal{R}_P$ are arrows $f:A\to B$ such that $\rho\leq_{A\times A}P(f\times f)(\sigma)$. Composition and identities are computed in $\CC$. The product in the category $\mathcal{R}_P$ of a pair of objects $(A,\rho)$ and $(B,\sigma)$ is $(A\times B,\rho\boxtimes\sigma)$, together with the projections from $A\times B$ to $A$ and $B$. Then, the elementary quotient completion $(P)_q:\mathcal{R}_P\op\to\Pos$ will be given by $(P)_q(A,\rho)=\des_\rho=\{\alpha\in P(A)\mid P(\pr1)(\alpha)\land\rho\leq P(\pr2)(\alpha)\}$, and $(P)_q( f )=P(f)$. Finite meets in $\des_\rho$ are computed in $P(A)$. The fibered equality in $(P)_q\big((A,\rho)\times(A,\rho)\big)=\des_{\rho\boxtimes\rho}$ is $\rho$ itself. In particular $(P)_q$ is an elementary doctrine.

The elementary quotient completion comes with a 1-arrow $(J,j):P\to(P)_q$ in $\ED$, where the functor $J:\CC\to\mathcal{R}_P$ maps a homomorphism $f:A\to B$ to $f:(A,\delta_A)\to(B,\delta_B)$, and each component $j_A:P(A)\to (P)_q(A,\delta_A)$ is the identity of $P(A)$.

A $P$-quotient of a $P$-equivalence relation $\rho$ on $A$ is a homomorphism $q:A\to C$ in $\CC$ such that $\rho\leq P(q\times q)(\delta_C)$ and, for every homomorphism $g:A\to B$ satisfying $\rho\leq P(g\times g)(\delta_B)$, there exists a unique homomorphism $h:C\to B$ such that $g=hq$. A homomorphism $f: A\to B$ in $\CC$ is descent if the functor $P(f):P(B)\to P(A)$ is full. 

Then, let $\QED$ be the 2-full 2-subcategory of $\ED$ whose objects are elementary doctrines $P:\CC\op\to\Pos$ in which every $P$-equivalence relation has a $P$-quotient that is a descent homomorphism; the 1-arrows are those arrows $(G,\mathfrak g):P\to Z$ in $\ED$ such that $G$ preserves quotients---meaning, if $q:A\to C$ is a quotient of a $P$-equivalence relation $\rho$ on $A$, then $Gq$ is a quotient of the $Z$-equivalence relation $\mathfrak{g}_{A\times A}(\rho)$ on $GA$.

The doctrine $(P)_q$ and the homomorphism $(J,j)$ have the following universal property, stated in Theorem 4.5 of \cite{ElemQuotCompl}: there is an equivalence of categories 
\begin{equation*}\blank\circ(J,j):\QED((P)_q,Z) \to \ED(P,Z)\end{equation*}
for every $Z$ in $\QED$.
Having concluded the recap, we can resume our proof.

Since $\mathfrak l_{D\times D}(\delta_D)$ is an equivalence relation on $(\Lan_FH)(D)$, we can define a functor $\mathscr{L}=\langle\Lan_FH(\blank),\mathfrak l_{\blank\times\blank}(\delta_\blank)\rangle:\mathbb{D}\to\mathcal{R}_{\Sub}$, where $\mathcal{R}_{\Sub}$ is the category of equivalence relations of $\Sub:\mathbb{E}\op\to\Pos$. This follows from the fact that a $\Sub$-equivalence relation is an equivalence relation in the usual meaning. Given an arrow $g:D'\to D$ in $\mathbb{D}$, we define
\begin{equation*}\mathscr{L}(g)=(\Lan_FH)(g) :\big((\Lan_FH)(D'),\mathfrak l_{D'\times D'}(\delta_{D'})\big)\to\big((\Lan_FH)(D),\mathfrak l_{D\times D}(\delta_D)\big).\end{equation*}
\pointinproof{$\mathscr{L}$ is well defined on arrows}
In order to prove this is well defined we need
\begin{equation*}\mathfrak l_{D'\times D'}(\delta_{D'})\leq \big((\Lan_FH)(g)\times(\Lan_FH)(g)\big)^*(\mathfrak l_{D\times D}(\delta_D))\end{equation*}
in $\Sub((\Lan_FH)(D\times D))$, so we need a map $\dom\mathfrak l_{D'\times D'}(\delta_D')\to\dom\mathfrak l_{D\times D}(\delta_D)$ making the external diagram commute.
\[\begin{tikzcd}
	{\dom\mathfrak{l}_{D'\times D'}(\delta_{D'})} \\
	& {\dom((\Lan_FH)(g\times g))^*(\mathfrak{l}_{D\times D}(\delta_{D}))} & {\dom\mathfrak{l}_{D\times D}(\delta_{D})} \\
	& {(\Lan_FH)(D'\times D')} & {(\Lan_FH)(D\times D)}
	\arrow["((\Lan_FH)(g\times g))^*(\mathfrak{l}_{D\times D}(\delta_{D}))"'{yshift=0.3ex},tail, from=2-2, to=3-2]
	\arrow["{(\Lan_FH)(g\times g)}", from=3-2, to=3-3]
	\arrow[from=2-2, to=2-3]
	\arrow["\mathfrak{l}_{D\times D}(\delta_{D})",tail, from=2-3, to=3-3]
	\arrow["\mathfrak{l}_{D'\times D'}(\delta_{D'})"',bend right, tail, from=1-1, to=3-2]
	\arrow[bend left=10pt, from=1-1, to=2-3]
	\arrow["\leq",dashed, tail, from=1-1, to=2-2]
	\arrow["\lrcorner"{anchor=center, pos=0.125}, draw=none, from=2-2, to=3-3]
\end{tikzcd}\]
So now consider for each $(K,\mathfrak k):R\to\Sub$ and $\theta:(H,\mathfrak h)\to(K,\mathfrak k)(F,\mathfrak f)$.
\[\begin{tikzcd}
	& {KD'} \\
	{\dom\widehat\theta_{D'\times D'}^*(\Delta_{KD'})} & {\dom\widehat\theta_{D\times D}^*(\Delta_{KD})} & KD \\
	{(\Lan_FH)(D'\times D')} & {(\Lan_FH)(D\times D)} & {KD\times KD}
	\arrow["{\widehat\theta_{D\times D}^*(\Delta_{KD})}", tail, from=2-2, to=3-2]
	\arrow["{\widehat\theta_{D\times D}}"', from=3-2, to=3-3]
	\arrow["{\widehat\theta_{{D\times D}_{\mid\dom\widehat\theta_{D\times D}^*(\Delta_{KD})}}}"{xshift=-2ex}, from=2-2, to=2-3]
	\arrow["{\Delta_{KD}}", tail, from=2-3, to=3-3]
	\arrow["\lrcorner"{anchor=center, pos=0.125}, draw=none, from=2-2, to=3-3]
	\arrow["{(\Lan_FH)(g\times g)}"'{yshift=-0.8ex}, from=3-1, to=3-2]
	\arrow[dashed, from=2-1, to=2-2]
	\arrow["{\widehat\theta_{D'\times D'}^*(\Delta_{KD'})}"', tail, from=2-1, to=3-1]
	\arrow["{K(g\times g)\widehat\theta_{D'\times D'}}"', bend right, from=3-1, to=3-3]
	\arrow["{\widehat\theta_{{D'\times D'}_{\mid\dom\widehat\theta_{D'\times D'}^*(\Delta_{KD'})}}}"{xshift=-1ex, yshift=-1.5ex},bend left=20pt, from=2-1, to=1-2]
	\arrow["Kg", bend left=20pt, from=1-2, to=2-3, end anchor={[xshift=2ex]}]
\end{tikzcd}\]
By definition of $\mathfrak{l}$ as infimum we can find the wanted arrow from $\dom\mathfrak l_{D'\times D'}(\delta_{D'})$ to $\dom\mathfrak l_{D\times D}(\delta_D)$.
\pointinproof{$\mathscr{L}$ preserves products}
To see this, compute $\mathscr{L}(D\times D')$ and $\mathscr{L}(D)\times\mathscr{L}(D')$: the first projections $\Lan_FH(D\times D')=\Lan_FH(D)\times\Lan_FH(D')$ coincide, since $\Lan_FH$ preserves products; so we need to show that also the equivalence relations $\mathfrak{l}_{D\times D'\times D\times D'}(\delta_{D\times D'})$ and $\mathfrak{l}_D(\delta_D)\boxtimes\mathfrak{l}_{D'}(\delta_{D'})$ are the same subobject of $(\Lan_FH)(D\times D'\times D\times D')$.
First of all, we have that:
\begin{equation}\label{eq:pres_prod1}\mathfrak{l}_{D\times D'\times D\times D'}(\delta_{D\times D'})=\bigwedge_{(K,\mathfrak k),\theta}{\widehat{\theta}_{D\times D'\times D\times D'}}^*(\Delta_{KD\times KD'});\end{equation}
on the other hand we have:
\begin{align}\label{eq:pres_prod2}\mathfrak{l}_D(\delta_D)\boxtimes\mathfrak{l}_{D'}(\delta_{D'})&=\ple{\pr1,\pr3}^*\left(\bigwedge_{(K,\mathfrak k),\theta}{\widehat{\theta}_{D\times D}}^*(\Delta_{KD})\right)\land\ple{\pr2,\pr4}^*\left(\bigwedge_{(K,\mathfrak k),\theta}{\widehat{\theta}_{D'\times D'}}^*(\Delta_{KD'})\right)\nonumber\\
&=\bigwedge_{(K,\mathfrak k),\theta}\left(\ple{\pr1,\pr3}^*{\widehat{\theta}_{D\times D}}^*(\Delta_{KD})\land\ple{\pr2,\pr4}^*{\widehat{\theta}_{D'\times D'}}^*(\Delta_{KD'})\right);\end{align}
we prove with some diagram computation that for each $(K,\mathfrak{k}),\theta$ the arguments in the meets are the same.
So we compute at first every ${\widehat{\theta}_{D\times D'\times D\times D'}}^*(\Delta_{KD\times KD'})$ appearing in \eqref{eq:pres_prod1}:
\begin{equation*}\begin{tikzcd}
	{\dom\widehat{\theta}^*_{D\times D'\times D\times D'}(\Delta_{KD\times KD'})} && {KD\times KD'} \\
	{(\Lan_FH)(D\times D'\times D\times D')} && {KD\times KD'\times KD\times KD'}
	\arrow["{\widehat{\theta}^*_{D\times D'\times D\times D'}(\Delta_{KD\times KD'})}"', tail, from=1-1, to=2-1]
	\arrow["{\widehat{\theta}_{{D\times D'\times D\times D'}_{\mid\dom\widehat{\theta}^*_{D\times D'\times D\times D'}(\Delta_{KD\times KD'})}}}"{yshift=0.7ex}, from=1-1, to=1-3]
	\arrow["{\Delta_{KD\times KD'}}", tail, from=1-3, to=2-3]
	\arrow[""{name=0, anchor=center, inner sep=0}, "{\widehat{\theta}_{D\times D'\times D\times D'}}"', from=2-1, to=2-3]
	\arrow["\lrcorner"{anchor=center, pos=0.125}, draw=none, from=1-1, to=0]
\end{tikzcd}.\end{equation*}
Then, for every $(K,\mathfrak{k}),\theta$, we compute $\ple{\pr1,\pr3}^*{\widehat{\theta}_{D\times D}}^*(\Delta_{KD})$:
\begin{equation}\label{diagr:2}\begin{tikzcd}[column sep=15pt]
	{\dom\ple{\pr1,\pr3}^*{\widehat{\theta}_{D\times D}}^*(\Delta_{KD})} && KD \\
	{(\Lan_FH)(D\times D'\times D\times D')} & {(\Lan_FH)(D\times D)} & {KD\times KD} \\
	& {KD\times KD'\times KD\times KD'}
	\arrow["{\ple{\pr1,\pr3}^*{\widehat{\theta}_{D\times D}}^*(\Delta_{KD})}"', tail, from=1-1, to=2-1]
	\arrow["{(\widehat{\theta}_{D\times D}\ple{\pr1,\pr3})_{\mid\dom\ple{\pr1,\pr3}^*{\widehat{\theta}_{D\times D}}^*(\Delta_{KD})}}", from=1-1, to=1-3]
	\arrow["{\Delta_{KD}}", tail, from=1-3, to=2-3]
	\arrow["{\ple{\pr1,\pr3}}"{pos=0.4}, from=2-1, to=2-2]
	\arrow["{\widehat{\theta}_{D\times D}}"{pos=0.6}, from=2-2, to=2-3]
	\arrow["\lrcorner"{anchor=center, pos=0.125}, draw=none, from=1-1, to=2-3]
	\arrow["{\widehat{\theta}_{D\times D'\times D\times D'}}"'{pos=0.3, xshift=4ex}, from=2-1, to=3-2]
	\arrow["{\ple{\pr1,\pr3}}"'{pos=0.7}, from=3-2, to=2-3]
\end{tikzcd},\end{equation}
and similarly $\ple{\pr2,\pr4}^*{\widehat{\theta}_{D'\times D'}}^*(\Delta_{KD'})$:
\begin{equation}\label{diagr:3}\begin{tikzcd}[column sep=9pt]
	{\dom\ple{\pr2,\pr4}^*{\widehat{\theta}_{D'\times D'}}^*(\Delta_{KD'})} && {KD'} \\
	{(\Lan_FH)(D\times D'\times D\times D')} & {(\Lan_FH)(D'\times D')} & {KD'\times KD'} \\
	& {KD\times KD'\times KD\times KD'}
	\arrow["{\ple{\pr2,\pr4}^*{\widehat{\theta}_{D'\times D'}}^*(\Delta_{KD'})}"', tail, from=1-1, to=2-1]
	\arrow["{(\widehat{\theta}_{D'\times D'}\ple{\pr2,\pr4})_{\mid\dom\ple{\pr2,\pr4}^*{\widehat{\theta}_{D'\times D'}}^*(\Delta_{KD'})}}", from=1-1, to=1-3]
	\arrow["{\Delta_{KD'}}", tail, from=1-3, to=2-3]
	\arrow["{\ple{\pr2,\pr4}}"{pos=0.4,yshift=0.2ex}, from=2-1, to=2-2]
	\arrow["{\widehat{\theta}_{D'\times D'}}"{pos=0.6}, from=2-2, to=2-3]
	\arrow["\lrcorner"{anchor=center, pos=0.125}, draw=none, from=1-1, to=2-3]
	\arrow["{\widehat{\theta}_{D\times D'\times D\times D'}}"'{pos=0.3, xshift=4ex}, from=2-1, to=3-2]
	\arrow["{\ple{\pr2,\pr4}}"'{pos=0.7}, from=3-2, to=2-3]
\end{tikzcd}.\end{equation}
Taking their pullback we get the components of \eqref{eq:pres_prod2}:
\[\begin{tikzcd}[column sep=0.9pt, row sep=huge]
	{\dom(\ple{\pr1,\pr3}^*{\widehat{\theta}_{D\times D}}^*(\Delta_{KD})\land\ple{\pr2,\pr4}^*{\widehat{\theta}_{D'\times D'}}^*(\Delta_{KD'}))} && {\dom\ple{\pr2,\pr4}^*{\widehat{\theta}_{D'\times D'}}^*(\Delta_{KD'})} \\
	{\dom\ple{\pr1,\pr3}^*{\widehat{\theta}_{D\times D}}^*(\Delta_{KD})} && {(\Lan_FH)(D\times D'\times D\times D')}
	\arrow["{\ple{\pr2,\pr4}^*{\widehat{\theta}_{D'\times D'}}^*(\Delta_{KD'})}"{description,xshift=3ex}, tail, from=1-3, to=2-3]
	\arrow["{\ple{\pr1,\pr3}^*{\widehat{\theta}_{D\times D}}^*(\Delta_{KD})}"', tail, from=2-1, to=2-3]
	\arrow["{\omega_2}"', tail, from=1-1, to=2-1, shift right=17]
	\arrow["{\omega_1}",curve={height=-20pt}, tail, from=1-1, to=1-3]
	\arrow["{\ple{\pr1,\pr3}^*{\widehat{\theta}_{D\times D}}^*(\Delta_{KD})\land\ple{\pr2,\pr4}^*{\widehat{\theta}_{D'\times D'}}^*(\Delta_{KD'})}"{description,xshift=-4ex}, tail, from=1-1, to=2-3, start anchor={[xshift=-7ex]}]
	\arrow["\lrcorner"{description, pos=0.2}, shift left=-8, draw=none, from=1-1, to=2-1]
\end{tikzcd}.\]
Now, for each $(K,\mathfrak{k}),\theta$, the inequality
\begin{equation*}{\widehat{\theta}_{D\times D'\times D\times D'}}^*(\Delta_{KD\times KD'})\leq\ple{\pr1,\pr3}^*{\widehat{\theta}_{D\times D}}^*(\Delta_{KD})\land\ple{\pr2,\pr4}^*{\widehat{\theta}_{D'\times D'}}^*(\Delta_{KD'})
\end{equation*}
holds if and only if both
\begin{equation}\label{eq:ineq1}{\widehat{\theta}_{D\times D'\times D\times D'}}^*(\Delta_{KD\times KD'})\leq\ple{\pr1,\pr3}^*{\widehat{\theta}_{D\times D}}^*(\Delta_{KD})\end{equation}
and
\begin{equation}\label{eq:ineq2}{\widehat{\theta}_{D\times D'\times D\times D'}}^*(\Delta_{KD\times KD'})\leq\ple{\pr2,\pr4}^*{\widehat{\theta}_{D'\times D'}}^*(\Delta_{KD'})\end{equation}
hold.
To show the inequality \eqref{eq:ineq1}, take the pair $\widehat{\theta}^*_{D\times D'\times D\times D'}(\Delta_{KD\times KD'})$ and the first projection of
\begin{equation*}\widehat{\theta}_{{D\times D'\times D\times D'}_{\mid\dom\widehat{\theta}^*_{D\times D'\times D\times D'}(\Delta_{KD\times KD'})}},\end{equation*} then use the universal property of the pullback \eqref{diagr:2} above. Similarly, prove the inequality \eqref{eq:ineq2} by taking the pair $\widehat{\theta}^*_{D\times D'\times D\times D'}(\Delta_{KD\times KD'})$ and the second projection of
\begin{equation*}\widehat{\theta}_{{D\times D'\times D\times D'}_{\mid\dom\widehat{\theta}^*_{D\times D'\times D\times D'}(\Delta_{KD\times KD'})}},\end{equation*} then use the universal property of the pullback \eqref{diagr:3} above.
For the converse direction, take the pair $\ple{\pr1,\pr3}^*{\widehat{\theta}_{D\times D}}^*(\Delta_{KD})\land\ple{\pr2,\pr4}^*{\widehat{\theta}_{D'\times D'}}^*(\Delta_{KD'})$ and the arrow that has as a first component 
\begin{equation*}(\widehat{\theta}_{D\times D}\ple{\pr1,\pr3})_{\mid\dom\ple{\pr1,\pr3}^*{\widehat{\theta}_{D\times D}}^*(\Delta_{KD})}\omega_2\end{equation*}
and as second component
\begin{equation*}(\widehat{\theta}_{D'\times D'}\ple{\pr2,\pr4})_{\mid\dom\ple{\pr2,\pr4}^*{\widehat{\theta}_{D'\times D'}}^*(\Delta_{KD'})}\omega_1,\end{equation*}
then use the universal property of the pullback \eqref{diagr:2}.
\pointinproof{$(\mathscr{L},\mathfrak l)$ is well defined}
We want to complete the following diagram with $\mathfrak{l}$:
\begin{center}
\begin{tikzcd}
\mathbb{D}\op\arrow[rr,"\mathscr{L}\op"] \arrow[dr,"R"' ,""{name=L}]&&\mathcal{R}_{\Sub}\op\arrow[dl,"(\Sub)_q" ,""'{name=R}]\\
&\Pos\arrow[rightarrow,"\mathfrak{l}","\cdot"', from=L, to=R, bend left=10]
\end{tikzcd}.
\end{center}
We prove that $\mathfrak{l}:R\xrightarrow{\cdot}(\Sub)_q\mathscr{L}\op$ is well defined by showing that for each $\gamma\in RD$, we have $\mathfrak{l}_D(\gamma)\in\des_{\mathfrak{l}_{D\times D}(\delta_D)}$, i.e.\ $\pr1^*\mathfrak{l}_D(\gamma)\land\mathfrak{l}_{D\times D}(\delta_D)\leq\pr2^*\mathfrak{l}_D(\gamma)$ in $\Sub(\Lan_FH)(D\times D)$.
Indeed
\begin{align*}\pr1^*\left(\bigwedge_{(K,\mathfrak k),\theta}\widehat{\theta}_D^*(\mathfrak k _D(\gamma))\right)&\land\bigwedge_{(K,\mathfrak k),\theta}{(\widehat{\theta}_{D}\times\widehat{\theta}_{D})}^*(\Delta_{KD})\\
&=\bigwedge_{(K,\mathfrak k),\theta}\left(\pr1^*\widehat{\theta}_D^*(\mathfrak k _D(\gamma))\land{(\widehat{\theta}_{D}\times\widehat{\theta}_{D})}^*(\Delta_{KD})\right)\\
&=\bigwedge_{(K,\mathfrak k),\theta}{(\widehat{\theta}_{D}\times\widehat{\theta}_{D})}^*\left(\pr1^*(\mathfrak k _D(\gamma))\land\Delta_{KD}\right)\\
&\leq\bigwedge_{(K,\mathfrak k),\theta}{(\widehat{\theta}_{D}\times\widehat{\theta}_{D})}^*\left(\pr2^*(\mathfrak k _D(\gamma))\right)=\pr2^*\mathfrak{l}_D(\gamma).\end{align*}
So we proved that we have a 1-arrow between the doctrines $(\mathscr{L},\mathfrak l):R\to(\Sub)_q$.
\pointinproof{$(\mathscr{L},\mathfrak l)$ is in $\ED$} By construction $\mathfrak{l}$ is a natural transformation and preserves finite meets. Recalling that in the elementary quotient completion $(P)_q:\mathcal{R}_P\op\to\Pos$ of a doctrine $P:\CC\op\to\Pos$, the fibered equality in $(P)_q\big((A,\rho)\times(A,\rho)\big)=\des_{\rho\boxtimes\rho}$ is $\rho$ itself, we apply this to the case $(\Sub)_q$ to obtain that
\begin{equation*}\mathfrak{l}_{D\times D}:R(D\times D)\to(\Sub)_q\big(((\Lan_FH)(D),\mathfrak l_{D\times D}(\delta_D))\times((\Lan_FH)(D),\mathfrak l_{D\times D}(\delta_D))\big)\end{equation*}
or, computing the codomain,
\begin{equation*}\mathfrak{l}_{D\times D}:R(D\times D)\to\des_{\mathfrak l_{D\times D}(\delta_D)\boxtimes\mathfrak l_{D\times D}(\delta_D)}\end{equation*}
so $\mathfrak{l}$ preserves the fibered equality and $(\mathscr{L},\mathfrak l)$ is in $\ED$.
\pointinproof{From $(\Sub)_q$ to $\Sub$} Recall that our goal is to define for each $(H,\mathfrak{h}):P\to\Sub$ in $\ED$ a suitable 1-arrow from $R$ to $\Sub$. So we look for an arrow $(Q,\mathfrak{q}):(\Sub)_q\to\Sub$ in order to define the wanted map as the composition $(Q,\mathfrak{q})(\mathscr{L},\mathfrak{l})$.
\begin{center}
\begin{tikzcd}
\mathcal{R}_{\Sub}\op\arrow[rr,"Q\op"] \arrow[dr,"(\Sub)_q"' ,""{name=L}]&&\mathbb{E}\op\arrow[dl,"\Sub" ,""'{name=R}]\\
&\Pos\arrow[rightarrow,"\mathfrak{q}","\cdot"', from=L, to=R, bend left=10]
\end{tikzcd}
\end{center}
To do this, we want to use the universal property of $(\Sub)_q$: applying Theorem 4.5 of \cite{ElemQuotCompl} recalled in the recap above to our case, there is an equivalence of categories 
\begin{equation*}\blank\circ(J,j):\QED((\Sub)_q,Z) \to \ED(\Sub,Z)\end{equation*}
for every $Z$ in $\QED$.

So we prove that $\Sub$ is in $\QED$, and define $(Q,\mathfrak{q})$ as the essentially unique 1-arrow such that $(Q,\mathfrak{q})(J,j)=\id{\Sub}$.
We show that every equivalence relation $\ple{s_1,s_2}:S\rightarrowtail X\times X$ in $\mathbb{E}$ has a quotient that is a descent homomorphism. Since Grothendieck topos are cocomplete, the quotient exists, and it is the coequalizer $q:X\to X/S$ of $s_1$ and $s_2$. Moreover, since the pullback functor $q^*:\Sub(X/S)\to\Sub(X)$ along an epimorphism reflects the order, the quotient $q$ is a descent homomorphism, as claimed.

\begin{claim}\label{claim:left_adj}
The assignment defined above, sending $(H,\mathfrak{h})$ to $(Q,\mathfrak{q})(\mathscr{L},\mathfrak{l})$ extends to a left adjoint to $\blank\circ(F,\mathfrak f)$.
\end{claim}
We look for the universal arrow
\begin{equation*}\eta_{(H,\mathfrak h)}:(H,\mathfrak h)\to(Q,\mathfrak{q})(\mathscr{L},\mathfrak{l})(F,\mathfrak f).\end{equation*}
In particular, we need a natural transformation $\eta_{(H,\mathfrak h)}:H\xrightarrow{\cdot} Q\mathscr{L}F$, or equivalently---by the properties of the left Kan extension---a factorization for any object $A\in\CC$
\begin{equation*}\big(\eta_{(H,\mathfrak h)}\big)_A:HA\xrightarrow{\mu_A}(\Lan_FH)(FA)\xrightarrow{\rho_{FA}}Q\big((\Lan_FH)(FA),\mathfrak l_{FA\times FA}(\delta_{FA})\big)\end{equation*}
for some natural transformation $\rho:\Lan_FH\xrightarrow{\cdot}Q\mathscr{L}$. To define $\rho$, note that for any $(X,s)\in~\mathcal{R}_{\Sub}$, there is an arrow $\id{X}:(X,\Delta_X)\to(X,s)$ in $\mathcal{R}_{\Sub}$---since $s$ is an equivalence relation. Apply $Q$ to obtain $Q(\id{X}):X\to Q(X,s)$ in $\mathbb{E}$. So we can define $\rho_D=Q(\id{(\Lan_FH)(D)})$; this is clearly a natural transformation, since the naturality square
\[\begin{tikzcd}
	{D'} & {(\Lan_FH)(D')} && {Q((\Lan_FH)(D'),\mathfrak l_{D'\times D'}(\delta_{D'}))} \\
	D & {(\Lan_FH)(D)} && {Q((\Lan_FH)(D),\mathfrak l_{D\times D}(\delta_{D}))}
	\arrow["g"', from=1-1, to=2-1]
	\arrow["{(\Lan_FH)(g)}"', from=1-2, to=2-2]
	\arrow["{Q(\id{(\Lan_FH)(D)})}"'{yshift=-0.2ex}, from=2-2, to=2-4]
	\arrow["{Q(\id{(\Lan_FH)(D')})}", from=1-2, to=1-4]
	\arrow["{Q(\Lan_FH)(g)}", from=1-4, to=2-4]
\end{tikzcd}\]
is the image through $Q$ of
\[\begin{tikzcd}
	{((\Lan_FH)(D'),\Delta_{(\Lan_FH)(D')})} && {((\Lan_FH)(D'),\mathfrak l_{D'\times D'}(\delta_{D'}))} \\
	{((\Lan_FH)(D),\Delta_{(\Lan_FH)(D)})} && {((\Lan_FH)(D),\mathfrak l_{D\times D}(\delta_{D}))}
	\arrow["{(\Lan_FH)(g)}"', from=1-1, to=2-1]
	\arrow["{\id{(\Lan_FH)(D)}}"', from=2-1, to=2-3]
	\arrow["{\id{(\Lan_FH)(D')}}", from=1-1, to=1-3]
	\arrow["{(\Lan_FH)(g)}", from=1-3, to=2-3]
\end{tikzcd}.\]
So $\eta_{(H,\mathfrak{h})}$ is defined, and it is a natural transformation, since it is the composition of natural transformations.
\pointinproof{$\eta_{(H,\mathfrak h)}$ is a 2-arrow} We need to show that
\begin{equation*}\mathfrak{h}_A(\alpha)\leq \big(\eta_{(H,\mathfrak h)}\big)_A^*(\mathfrak{q}_{\mathscr{L}FA}\mathfrak{l}_{FA}\mathfrak{f}_A(\alpha))=\mu_A^*\rho_{FA}^*(\mathfrak{q}_{\mathscr{L}FA}\mathfrak{l}_{FA}\mathfrak{f}_A(\alpha)).\end{equation*}
Observe that from naturality of $\mathfrak{q}$, we have for any $(X,s)\in\mathcal{R}_{\Sub}$ the following commutative diagram:
\[\begin{tikzcd}
	{(X,\Delta_X)} & {(\Sub)_q(X,\Delta_X)} & {\Sub Q(X,\Delta_X)} \\
	{(X,s)} & {(\Sub)_q(X,s)} & {\Sub Q(X,s)}
	\arrow["{\id{X}}", from=1-1, to=2-1]
	\arrow["{\mathfrak{q}_{(X,\Delta_X)}}"{yshift=0.2ex}, from=1-2, to=1-3]
	\arrow["{\mathfrak{q}_{(X,s)}}"', from=2-2, to=2-3]
	\arrow["{\id{X}^*}", from=2-2, to=1-2]
	\arrow["{Q(\id{X})^*}"', from=2-3, to=1-3]
\end{tikzcd}\]
Note that $\id{X}^*$ is just the inclusion $\des_s\subseteq\Sub X$, and $\mathfrak{q}_{(X,\Delta_X)}=\mathfrak{q}_{JX}$ is such that $\mathfrak{q}_{JX}j_X=\id{\Sub X}$, hence $\mathfrak{q}_{(X,\Delta_X)}=\id{\Sub X}$, so to conclude we have that $Q(\id{X})^*\mathfrak{q}_{(X,s)}$ is the inclusion $\des_s\subseteq\Sub X$. Apply this when
\begin{equation*}(X,s)=((\Lan_FH)(FA),\mathfrak{l}_{FA\times FA}(\delta_{FA}))=\mathscr{L}(FA)\end{equation*}
to get that $\rho_{FA}^*\mathfrak{q}_{\mathscr{L}FA}$ acts as the identity, so our claim becomes
\begin{equation*}\mathfrak{h}_A(\alpha)\leq \mu_A^*(\mathfrak{l}_{FA}\mathfrak{f}_A(\alpha)).\end{equation*}
But now
\begin{equation*}\mu_A^*(\mathfrak{l}_{FA}\mathfrak{f}_A(\alpha))=\mu_A^*\big(\bigwedge_{(K,\mathfrak k),\theta}\widehat{\theta}_{FA}^*(\mathfrak k _{FA}\mathfrak{f}_A(\alpha))\big)=\bigwedge_{(K,\mathfrak k),\theta}\mu_A^*\widehat{\theta}_{FA}^*(\mathfrak k _{FA}\mathfrak{f}_A(\alpha))=\bigwedge_{(K,\mathfrak k),\theta}\theta_A^*(\mathfrak k _{FA}\mathfrak{f}_A(\alpha)),\end{equation*}
but every $\theta$ is a 2-arrow, so $\mathfrak{h}_A(\alpha)\leq\theta_A^*(\mathfrak k _{FA}\mathfrak{f}_A(\alpha))$, hence $\eta_{(H,\mathfrak h)}$ is indeed a 2-arrow.
\pointinproof{$\eta_{(H,\mathfrak h)}$ has the universal property} We now prove that for every arrow $(K,\mathfrak k):R\to\Sub$ in $\ED$ and every 2-arrow $\theta:(H,\mathfrak h)\to(K,\mathfrak k)(F,\mathfrak f)$ there exists a unique 2-arrow
\begin{equation*}\overline\theta:(Q,\mathfrak{q})(\mathscr{L},\mathfrak{l})\to(K,\mathfrak k),\end{equation*}
making the following diagram commute:
\[\begin{tikzcd}
	{(H,\mathfrak h)} && {(Q,\mathfrak{q})(\mathscr{L},\mathfrak{l})(F,\mathfrak f)} \\
	& {(K,\mathfrak k)(F,\mathfrak f)}
	\arrow["{\eta_{(H,\mathfrak h)}}", from=1-1, to=1-3]
	\arrow["\theta"', from=1-1, to=2-2]
	\arrow["{\overline\theta\circ(F,\mathfrak f)}", from=1-3, to=2-2]
\end{tikzcd}.\]
For any object $D\in\mathbb{D}$, observe that the image of $\widehat\theta_D:\Lan_FHD\to KD$ in $\mathcal{R}_{\Sub}$ through $J$ factors as follows:
\[\begin{tikzcd}
	{(\Lan_FHD,\Delta_{\Lan_FHD})} & {(\Lan_FHD,\mathfrak{l}_{D\times D(\delta_D)})} & {(KD,\Delta_{KD})}
	\arrow["{\id{\Lan_FHD}}", from=1-1, to=1-2]
	\arrow["{\widehat\theta_D}", from=1-2, to=1-3]
	\arrow["{\widehat\theta_D}"', bend right=12pt, from=1-1, to=1-3]
\end{tikzcd}.\]
The first map trivially exists; the second one exists if and only if there exists a dotted arrow in the diagram below:
\[\begin{tikzcd}
	{\dom \mathfrak{l}_{D\times D}(\delta_D)} \\
	& {\dom \widehat\theta_{D\times D}^*(\Delta_{KD})} & {\Lan_FHD\times \Lan_FHD} \\
	& KD & {KD\times KD}
	\arrow["{\Delta_{KD}}"', from=3-2, to=3-3]
	\arrow["{\widehat\theta_D\times\widehat\theta_D}", from=2-3, to=3-3]
	\arrow[from=2-2, to=3-2]
	\arrow["{\widehat\theta_{D\times D}^*(\Delta_{KD})}"{yshift=0.4ex}, from=2-2, to=2-3]
	\arrow["{\mathfrak{l}_{D\times D}(\delta_D)}", bend left=10, from=1-1, to=2-3]
	\arrow[bend right=10, dashed, from=1-1, to=3-2]
	\arrow["\lrcorner"{anchor=center, pos=0.01}, draw=none, from=2-2, to=3-3]
\end{tikzcd}\]
but it exists by definition of $\mathfrak{l}_{D\times D}(\delta_D)$ as the infimum of all subobject of the form $\widehat\theta_{D\times D}^*(\Delta_{KD})$.
Apply $Q$ to the factorization of $\widehat{\theta}_D$ above to obtain in $\mathbb{E}$:
\[\begin{tikzcd}
	{\Lan_FHD} & {Q(\Lan_FHD,\mathfrak{l}_{D\times D(\delta_D)})} & {KD}
	\arrow["{\rho_D}", from=1-1, to=1-2]
	\arrow["{Q(\widehat\theta_D)}", from=1-2, to=1-3]
	\arrow["{\widehat\theta_D}"', bend right=12pt, from=1-1, to=1-3, end anchor={[xshift=1ex, yshift=-1ex]},]
\end{tikzcd}.\]
Define $\overline\theta_D=Q(\widehat\theta_D):Q\mathscr{L}D\to KD$. For any given $g:D'\to D$ in $\mathbb{E}$, we have that the image through $Q$ of the square on the right in the diagram below gives naturality of $\overline\theta$:
\[\begin{tikzcd}
	{(\Lan_FHD',\Delta_{\Lan_FHD'})} & {(\Lan_FHD',\mathfrak{l}_{D'\times D'(\delta_{D'})})} & {(KD',\Delta_{KD'})} \\
	{(\Lan_FHD,\Delta_{\Lan_FHD})} & {(\Lan_FHD,\mathfrak{l}_{D\times D(\delta_D)})} & {(KD,\Delta_{KD})}
	\arrow["{\id{\Lan_FHD}}", from=2-1, to=2-2]
	\arrow["{\widehat\theta_D}", from=2-2, to=2-3]
	\arrow["{\widehat\theta_D}"', bend right=12pt, from=2-1, to=2-3]
	\arrow["{\Lan_FH(g)}"', from=1-1, to=2-1]
	\arrow["{\id{\Lan_FHD'}}"'{yshift=-0.2ex}, from=1-1, to=1-2]
	\arrow["{\Lan_FH(g)}", from=1-2, to=2-2]
	\arrow["{\widehat\theta_{D'}}"', from=1-2, to=1-3]
	\arrow["{K(g)}", from=1-3, to=2-3]
	\arrow["{\widehat\theta_{D'}}", bend left=12pt, from=1-1, to=1-3]
\end{tikzcd}.\]
Now, to prove that $\overline\theta$ is a 2-arrow we show that for any object $D\in\mathbb{D}$ and any $\gamma\in RD$
\begin{equation*}\mathfrak{q}_{\mathscr{L}D}\mathfrak{l}_D(\gamma)\leq\overline\theta_D^*(\mathfrak{k}_D(\gamma)).\end{equation*}
Since $\mathfrak{l}_D(\gamma)\leq\widehat\theta^*_D(\mathfrak{k}_D(\gamma))$, it is enough to prove that $\mathfrak{q}_{\mathscr{L}D}\widehat\theta^*_D(\mathfrak{k}_D(\gamma))\leq Q(\widehat\theta_D)^*(\mathfrak{k}_D(\gamma))$. If $\rho_D^*$ is full, the last inequality is equivalent to $\rho_D^*\mathfrak{q}_{\mathscr{L}D}\widehat\theta^*_D(\mathfrak{k}_D(\gamma))\leq \rho_D^*Q(\widehat\theta_D)^*(\mathfrak{k}_D(\gamma))$, i.e.\ $\widehat\theta^*_D(\mathfrak{k}_D(\gamma))\leq Q(\id{\Lan_FHD})^*Q(\widehat\theta_D)^*(\mathfrak{k}_D(\gamma))=\widehat\theta_D^*\mathfrak{k}_D(\gamma)$. So we check the following:
\begin{claim}
The arrow $\rho_D$ is a regular epimorphism.\end{claim}
\begin{proof}
For any $(X,s=\ple{s_1,s_2}:S\rightarrowtail X\times X)\in\mathcal{R}_{\Sub}$, consider $\id{X}:(X,\Delta_X)\to(X,s)$ in $\mathcal{R}_{\Sub}$ and $Q(\id{X}):X\to Q(X,s)$ in $\mathbb{E}$. We prove that $Q(\id{X})$ is a regular epimorphism.

Note that given any $(X,r)\in\mathcal{R}_{\Sub}$, an $\mathcal{R}_{\Sub}$-equivalence relation on $(X,r)$ is an element $s$ in $(\Sub)_q((X,r)\times(X,r))=\des_{r\boxtimes r}\subseteq\Sub(X\times X)$ that is an equivalence relation on $X$ such that $r\leq s$. It follows that $\id{X}:(X,r)\to(X,s)$ is an arrow in $\mathcal{R}_{\Sub}$.

Moreover, it is an $\mathcal{R}_{\Sub}$-quotient of $s$: it is an arrow such that $s\leq(\id{}\times\id{})^*(s)=s$, and for every homomorphism $g:(X,r)\to(Y,u)$---i.e.\ $g:X\to Y$ such that $r\leq(g\times g)^*(u)$---such that $s\leq(g\times g)^*(u)$, we find a unique homomorphism $h:(X,r)\to(Y,u)$ such that $g=h\id{}$, indeed $h=g$, and it is an arrow in $\mathcal{R}_{\Sub}$ since $s\leq(g\times g)^*(u)$.
If we take $r=\Delta_X$, we have that $\id{X}:(X,\Delta_X)\to(X,s)$ is an $\mathcal{R}_{\Sub}$-quotient of $s$, hence $Q(\id{X}):X\to Q(X,s)$ is a $\Sub$-quotient of a $\Sub$-equivalence relation, hence it is a regular epimorphism.\end{proof}
Now we check commutativity of the triangle for the universal property:
\begin{equation*}\overline\theta_{FA}\big(\eta_{(H,\mathfrak h)}\big)_A=\overline\theta_{FA}\rho_{FA}\mu_A=\widehat\theta_{FA}\mu_A=\theta_A.\end{equation*}
To conclude, we show that $\overline\theta$ is unique. Suppose we have a 2-arrow $\lambda:(Q,\mathfrak{q})(\mathscr{L},\mathfrak{l})\to(K,\mathfrak k)$ making the triangle commute. In particular, for any object $A\in\CC$ we have $\lambda_{FA}\rho_{FA}\mu_A=\theta_A$ and for any $D\in\mathbb{D}, \gamma\in RD$ we have $\mathfrak{q}_{\mathscr{L}D}\mathfrak{l}_D(\gamma)\leq\lambda_D^*(\mathfrak{k}_D(\gamma))$. Consider the natural transformation $\lambda\circ\rho:\Lan_FH\xrightarrow{\cdot} K$. It is such that $(\lambda\circ\rho)\circ\mu=\theta$, so by the universal property of $\mu$ we have $\lambda\circ\rho=\widehat\theta$, but then we have $\widehat\theta_D=\lambda_D\rho_D=\overline\theta_D\rho_D$, so that $\lambda_D=\overline\theta_D$.

This concludes the proof of Claim \ref{claim:left_adj}, hence of Theorem \ref{thm:left_adj_prec}.

\section{Examples}
Before we give some examples, we prove a general result for some first-order theories. This generalizes Example 2.5.a of \cite{ElemQuotCompl}.

We refer to \cite{caramello} for the definitions about first-order calculus. We slightly change the doctrine of well-formed formulae of Example \ref{ex:doctr} considering just the fragment of Horn logic. Moreover, a theory in this context is not a set of closed formulae, but is instead a set of Horn sequents over $\Sigma$. We write in this case $\efbf^{\Sigma}_{\mathbb{T}}:\ctx_\Sigma\op\to\Pos$ for the elementary doctrine of Horn formulae: the base category is the same defined in Example \ref{ex:doctr}, each list of variable is sent to the poset reflection of Horn formulae---defined inductively as the smallest set containing relations, equalities, true constant and conjunctions of formulae---ordered by provable consequence in $\mathbb{T}$; reindexing are again defined as substitutions.
\begin{proposition}\label{prop:ed_mod}
Let ${\Sigma}$ be a first-order language and $\mathbb{T}$ a first-order theory in the language ${\Sigma}$ such that its axioms are Horn sequents. Then there exists an equivalence of categories:
\begin{equation*}\ED(\efbf^{\Sigma}_{\mathbb{T}},\Sub)\cong\Mod^{{\Sigma}}_{\mathbb T}\end{equation*}
where $\efbf^{\Sigma}_{\mathbb{T}}:\ctx_{\Sigma}\op\to\Pos$ is the elementary doctrine of Horn formulae in the language ${\Sigma}$ of the theory ${\mathbb T}$, $\Sub:\CC\op\to\Pos$ is the elementary doctrine of subobject for a given category $\CC$ with finite limits, and $\Mod^{\Sigma}_{\mathbb T}$ is the category whose objects are models of the theory $\mathbb{T}$ in the category $\CC$, and whose arrows are ${\Sigma}$-homomorphism.\end{proposition}
\begin{proof}
Since there is no confusion, we write $\efbf$ instead of $\efbf^{\Sigma}_{\mathbb{T}}$ and $\Mod$ instead of $\Mod^{\Sigma}_{\mathbb{T}}$.
For any given $\theta:(H,\mathfrak{h})\to(H',\mathfrak{h}')$ in $\ED(\efbf,\Sub)$, define $\theta_{(x)}:H(x)\to H'(x)$.

Observe that $H(x)$ is indeed a model of the theory $\mathbb{T}$: each $n$-ary function symbol $f$ in the language defines an arrow $f(x_1,\dots,x_n):(x_1,\dots,x_n)\to(x)$ in $\ctx$, hence its image through $H$---that preserves products---defines a map $f^H:\big(H(x)\big)^n\to H(x)$, which is the interpretation of $f$ in $H(x)$; each $n$-ary predicate symbol $R$ defines $R^H=\mathfrak{h}_{(x_1,\dots,x_n)}(R(x_1,\dots,x_n)):\dom(R^H)\rightarrowtail \big(H(x)\big)^n$. From now on, we write $\vec x$ instead of the list $(x_1,\dots,x_n)$.

Satisfiability of axioms follows by the fact that $\mathfrak{h}_{\vec x}(\alpha(\vec x))$ is the interpretation of $\alpha$ in $H(x)$ for each $\alpha(\vec x)\in\efbf(\vec x)$, and $\mathfrak{h}$ is monotone, so if we have an axiom $\alpha(\vec x)\vdash\beta(\vec x)$ in $\mathbb{T}$ we have $\alpha^H\leq\beta^H$, i.e.\ $H(x)$ satisfies $\alpha\vdash\beta$. To check that $\mathfrak{h}_{\vec x}(\alpha(\vec x))=\alpha^H$ we work recursively on the complexity of $\alpha$:
\begin{itemize}
\item $\alpha=\top$: $\alpha^H=\id{\big(H(x)\big)^n}=\mathfrak{h}_{\vec x}(\top)$ trivially holds;
\item $\alpha=R(\vec x)$: $\alpha^H=R^H=\mathfrak{h}_{\vec x}(R)$ by the definition given above;
\item $\alpha=\alpha_1\land\alpha_2$: $\alpha^H=\alpha^H_1\land\alpha^H_2=\mathfrak{h}_{\vec x}(\alpha_1)\land\mathfrak{h}_{\vec x}(\alpha_2)=\mathfrak{h}_{\vec x}(\alpha_1\land\alpha_2)$ since $\mathfrak{h}_{\vec x}$ preserves meets;
\item $\alpha=\big(t_1(\vec x)=t_2(\vec x)\big)$: $\alpha^H=\text{Eq}(t_1^H,t_2^H)$, the equalizer of the interpretations $t_1^H,t_2^H$ of the terms $t_1,t_2$ in $H(x)$.
\end{itemize}
It is left to prove then that if $\alpha=\big(t_1(\vec x)=t_2(\vec x)\big)$, we have $\alpha^H=\mathfrak{h}_{\vec{x}}\big(t_1(\vec{x})=t_2(\vec{x})\big)$. Naturality of $\mathfrak{h}$ with respect to the arrow $(t_1(\vec{x}),t_2(\vec{x})):\vec{x}\to(y_1,y_2)$ in $\ctx$, applied to the formula $\big(y_1=y_2\big)\in\efbf(y_1,y_2)$ gives:
\[\begin{tikzcd}
	{\efbf(y_1,y_2)} & {\Sub((H(x))^2)} \\
	{\efbf(\vec x)} & {\Sub((H(x))^n)}
	\arrow["{\vec t(\vec x)/\vec y}"', from=1-1, to=2-1]
	\arrow["{\mathfrak{h}_{\vec x}}"', from=2-1, to=2-2]
	\arrow["{\mathfrak{h}_{(y_1,y_2)}}", from=1-1, to=1-2]
	\arrow["{\ple{t_1^H,t_2^H}^*}", from=1-2, to=2-2]
\end{tikzcd}\]
in order to get:
\begin{equation*}\mathfrak{h}_{\vec{x}}\big(t_1(\vec{x})=t_2(\vec{x})\big)=\ple{t_1^H,t_2^H}^*(\Delta_{H(x)})=\text{Eq}(t_1^H,t_2^H),\end{equation*}
as claimed. This proves that the association $\ED(\efbf,\Sub)\to\Mod$ is well defined on objects.

Concerning arrows, first of all observe that since $H, H'$ preserve products and $\theta$ is a natural transformation, $\theta_{\vec{x}}=\theta_{(x)}\times\dots\times\theta_{(x)}$---$n$ times. So the naturality diagram of $\theta$ with respect to an arrow defined by an $n$-ary function symbol $f(\vec x):\vec x\to(x)$ gives the fact that $\theta_{(x)}$ preserves the interpretation of the function symbol $f$:
\begin{equation*}\theta_{(x)}f^H=f^{H'}\theta_{\vec{x}}=f^{H'}(\theta_{(x)}\times\dots\times\theta_{(x)});\end{equation*}
moreover, for any $n$-ary predicate symbol $R$, since $\theta$ is a $2$-arrow, we have
\begin{equation*}R^H=\mathfrak{h}_{\vec{x}}(R)\leq\theta_{\vec{x}}^{*}\big(\mathfrak{h}'_{\vec{x}}(R)\big)=\theta_{\vec{x}}^{*}R^{H'},\end{equation*}
so that $\theta_{(x)}$ is indeed a homomorphism in $\Mod$.

Now that the functor is well defined, we prove that it is full, faithful and essentially surjective. Faithfulness is trivial since, as seen above, each component of $\theta$ is uniquely determined by its component on the context with one variable.

Take now $g:H(x)\to H'(x)$ a homomorphism in $\Mod$, define $\theta^g_{\vec{x}}=g\times\dots\times g$---$|\vec x|$ times, where $|\vec x|$ is the length of the list. This defines a natural transformation $\theta:H\to H'$: naturality with respect to projections follows by definition, moreover for any function symbol the naturality square commutes since $g$ preserves interpretations, and then recursively since any other arrow is composition of projections and terms---defined by composition of function symbols---, naturality holds for any arrow in $\ctx$.
\[\begin{tikzcd}
	{(\vec x)} & {H(x)^{|\vec x|}} & {H'(x)^{|\vec x|}} \\
	{(x)} & {H(x)} & {H'(x)}
	\arrow["{t^H}"', from=1-2, to=2-2]
	\arrow["g"', from=2-2, to=2-3]
	\arrow["{g\times\dots\times g}", from=1-2, to=1-3]
	\arrow["{t^{H'}}", from=1-3, to=2-3]
	\arrow["{t(\vec x)}", from=1-1, to=2-1]
\end{tikzcd}\]
We show that $\theta^g$ is an arrow in $\ED(\efbf,\Sub)$, i.e.\ that for any $\alpha(\vec x)\in\efbf(\vec x)$ we have $\mathfrak{h}_{\vec{x}}\big(\alpha(\vec x)\big)\leq{\theta^g_{\vec{x}}}^{*}\big(\mathfrak{h}'_{\vec{x}}\big(\alpha(\vec x)\big)\big)$. Recursively on the complexity of $\alpha$ we observe that if $\alpha=\top$ or $\alpha=\beta\land\gamma$, the inequality holds since $\mathfrak{h},\mathfrak{h}'$ and ${\theta^g_{\vec{x}}}^{*}$ preserve the top element and meets; if $\alpha$ is an equality of terms $\alpha(\vec x)=\big(t_1(\vec{x})=t_2(\vec{x})\big)$, we show $\text{Eq}(t_1^H,t_2^H)\leq{\theta^g_{\vec{x}}}^{*}(\text{Eq}(t_1^{H'},t_2^{H'}))$, but this holds looking at the diagram below:
\[\begin{tikzcd}
	{\text{dom}(\alpha^{H})} \\
	\bullet & {\big(H(x)\big)^n} & {H(x)} \\
	{\text{dom}(\alpha^{H'})} & {\big(H'(x)\big)^n} & {H'(x)}
	\arrow["{\alpha^{H'}}"', tail, from=3-1, to=3-2]
	\arrow["{t_2^{H'}}"', shift right=1, from=3-2, to=3-3]
	\arrow["{t_1^{H'}}", shift left=1, from=3-2, to=3-3]
	\arrow["{\theta^g_{\vec{x}}}"', from=2-2, to=3-2]
	\arrow["{t_2^{H}}"', shift right=1, from=2-2, to=2-3]
	\arrow["g", from=2-3, to=3-3]
	\arrow["{t_1^{H}}", shift left=1, from=2-2, to=2-3]
	\arrow[from=2-1, to=3-1]
	\arrow["{{\theta^g_{\vec{x}}}^*(\alpha^{H'})}", tail, from=2-1, to=2-2]
	\arrow["\lrcorner"{anchor=center, pos=0.125}, draw=none, from=2-1, to=3-2]
	\arrow["{\alpha^{H}}", curve={height=-6pt}, tail, from=1-1, to=2-2]
	\arrow[curve={height=6pt}, dashed, from=1-1, to=2-1]
	\arrow[curve={height=18pt}, from=1-1, to=3-1]
\end{tikzcd}\]
the arrow $\dom(\alpha^H)\to\dom(\alpha^{H'})$ exists and makes the outer left square commute if and only if $t_1^{H'}\theta^g_{\vec{x}}\alpha^H=t_2^{H'}\theta^g_{\vec{x}}\alpha^H$, but this is true since $t_i^{H'}\theta^g_{\vec{x}}=gt_i^H$ for $i=1,2$. So the dashed arrow above exists by the universal property of the pullback, hence $\alpha^H\leq{\theta^g_{\vec{x}}}^*(\alpha^{H'})$, as claimed.
Finally, if $\alpha=R$ for some predicate symbol $R$, we have to check $R^H\leq {\theta^g_{\vec{x}}}^{*}(R^{H'})$, but this holds by definition of ${\Sigma}$-homomorphism.
So $\theta^g:(H,\mathfrak{h})\to(H',\mathfrak{h}')$ is well defined, and its image is $g$, so the functor is full.

To conclude, take $M$ a model of $\mathbb{T}$, and write $f^M:M^n\to M$ for the interpretation in $M$ of any $n$-ary function symbol $f$ in the language and $R^M:\dom(R^M)\rightarrowtail M^n$ for the interpretation of any $n$-ary predicate symbol $R$. We define a functor $H^M:\ctx\to\CC$ that maps $\vec x\mapsto M^{|\vec x|}$, projections in projections, $f(\vec x)\mapsto f^M$, and this trivially extends to lists of terms, defining a product preserving functor. Now define $\mathfrak{h}^M:\efbf\xrightarrow{\cdot}\Sub H\op$:
\begin{equation*}\mathfrak{h}^M_{\vec{x}}(\alpha(\vec x))=\alpha^M,\end{equation*}
the interpretation of $\alpha$ in $M$. It is well defined because $M$ is a model. By definition of interpretation, $\mathfrak{h}^M$ preserves top element, meet and fibered equality. To prove that it is a natural transformation, take a list of terms $\vec t(\vec x)=(t_1(\vec x),\dots,t_{|\vec y|}(\vec x)):\vec x\to\vec y$ and we prove that the following diagram is commutative:
\[\begin{tikzcd}
	{\efbf(\vec y)} & {\Sub(M^{|\vec y|})} \\
	{\efbf(\vec x)} & {\Sub(M^{|\vec x|})}
	\arrow["{\vec t(\vec x)/\vec y}"', from=1-1, to=2-1]
	\arrow["{\mathfrak{h}^M_{\vec{x}}}"', from=2-1, to=2-2]
	\arrow["{\mathfrak{h}^M_{\vec y}}", from=1-1, to=1-2]
	\arrow["{\ple{t_1^M,\dots,t_{|\vec y|}^M}^{*}}", from=1-2, to=2-2]
\end{tikzcd}\]
but this is true by definition of interpretation. Clearly $(H^M,\mathfrak{h}^M)\mapsto M$ so the functor is essentially surjective and defines the equivalence of categories.\end{proof}
\begin{example}\textbf{Some algebraic examples.}
We prove, using the equivalence of Proposition \ref{prop:ed_mod} in some specific theories, that many adjunction results in algebra can be obtained as a particular case of the adjunction shown in Theorem \ref{thm:left_adj_prec}.

Suppose we have an algebraic language ${\Sigma}$, and extend the language with some new function symbols to obtain a new algebraic language ${\Sigma}'$. Then suppose to extend the theory $\mathbb{T}$---which is a theory also in the language ${\Sigma}'$---with some new axioms of the form $\top\vdash \big(t(\vec x)=s(\vec x)\big)$, where $t$ and $s$ are terms in the language ${\Sigma}'$. Note that we could have ${\Sigma}={\Sigma}'$, so we can just extend the theory, or $\mathbb{T}=\mathbb{T}'$, so we just extend the language.
This extension can be translated in a homomorphism $(E,\mathfrak{e}):\efbf_{\mathbb{T}}^{{\Sigma}}\to\efbf^{{\Sigma'}}_{\mathbb{T}'}$ in $\ED$
\begin{center}
\begin{tikzcd}
\ctx_{{\Sigma}}\op\arrow[rr,"E\op"] \arrow[dr,"\efbf^{\Sigma}_{\mathbb{T}}"' ,""{name=L}]&&\ctx_{{\Sigma}'}\op\arrow[dl,"\efbf^{\Sigma'}_{{\mathbb{T}'}}" ,""'{name=R}]\\
&\Pos\arrow[rightarrow,"\mathfrak{e}","\cdot"', from=L, to=R, bend left=10]
\end{tikzcd}.
\end{center}
The functor $E$ is the inclusion of terms written in the language ${\Sigma}$ in the terms of the language ${\Sigma}'$; each component $\mathfrak{e}_{\vec{x}}$ of the natural transformation $\mathfrak{e}$ is the composition of the inclusion of $\efbf^{\Sigma}_{\mathbb{T}}(\vec x)$ in the poset $\efbf^{\Sigma'}_{\mathbb{T}}(\vec x)$ of Horn formulae in the extended language with respect to the same theory, with the quotient from $\efbf^{\Sigma'}_{\mathbb{T}}(\vec x)$ into $\efbf^{\Sigma'}_{\mathbb{T'}}(\vec x)$, that sends the equivalence class of a formula---with respect to reciprocal provability in the theory $\mathbb{T}$---to the equivalence class of the same formula, with respect to reciprocal provability in the theory $\mathbb{T}'$.
In any such extension, we have the following commutative diagram:
\[\begin{tikzcd}
	{\Mod^{\Sigma'}_{\mathbb T'}} & {\Mod^{\Sigma}_{\mathbb T}} \\
	{\ED(\efbf^{\Sigma'}_{\mathbb{T'}},\pws )} & {\ED(\efbf^{\Sigma}_{\mathbb{T}},\pws )}
	\arrow["\rotatebox{-90}{$\cong$}"{description}, draw=none, from=2-1, to=1-1]
	\arrow[from=1-1, to=1-2]
	\arrow["{\blank\circ(E,\mathfrak{e})}", from=2-1, to=2-2]
	\arrow["\rotatebox{-90}{$\cong$}"{description}, draw=none, from=2-2, to=1-2]
\end{tikzcd}\]
where $\pws :\Set_{\ast}\op\to\Pos$ is the elementary doctrine of subsets removing the empty set from the base category, and the arrow between the categories of models is the functor that forgets both the added structure from ${\Sigma'}$ that is not in ${\Sigma}$ and the axioms in $\mathbb{T'}$ that are not in $\mathbb{T}$.
So the left adjoint to the precomposition with $(E,\mathfrak e)$ described in the first section generalizes all such adjunctions in algebra.
Some examples include the adjunction between: sets and pointed sets, groups and abelian groups, monoids and semigroups, non-unitary rings and unitary rings, and so on.
\end{example}
\begin{example}\textbf{Extension and restriction of scalars.}
In a similar way, let $R,S$ be two commutative unitary rings, and let $a:R\to S$ be a ring homomorphism. One can obtain the category $R\mathbf{Mod}$ of modules over the ring $R$ as the category of models in the language ${\Sigma}=\{0, +, -\}\cup\{r\cdot\}_{r\in R}$---where $0$ is a constant, $+$ is a binary function symbol, and $-$ and each $r\cdot$ are unary function symbols---of the algebraic theory $\mathbb{T}$ with axioms making $\{0,+,-\}$ group operations, and $r\cdot$ the scalar multiplication with $r\in R$.

As seen above, we can define the equivalence $\ED(\efbf^{\Sigma}_{\mathbb{T}},\pws )\cong R\mathbf{Mod}$; similarly define the equivalence $\ED(\efbf^{\Sigma'}_{\mathbb{T'}},\pws )\cong S\mathbf{Mod}$. Here ${\Sigma'}$ and $\mathbb{T'}$ are not extension of ${\Sigma}$ and $\mathbb{T}$. However we can define a functor $E:\ctx_{\Sigma}\to\ctx_{\mathcal {L'}}$ that maps $0:()\to(x)$, $+:(x_1,x_2)\to(x)$,$-:(x)\to(x)$ in themselves, and each $r\cdot:(x)\to(x)$ in $a(r)\cdot:(x)\to(x)$; moreover define $\mathfrak{e}_{\vec{x}}:\efbf^{\Sigma}_{\mathbb{T}}(\vec x)\to\efbf^{\Sigma'}_{\mathbb{T'}}(\vec x)$, such that $\alpha(\vec x)\mapsto\alpha[a(r)/r](\vec x)$, meaning that each formula is sent essentially in itself, but each occurrence of $r$ in the terms that appear in $\alpha$ is substituted by $a(r)$, for every $r\in R$.

This function preserves trivially top element, meets and fibered equality, and defines a natural transformation.
The precomposition $\blank\circ(E,\mathfrak{e})$ recovers the adjunction between $R\mathbf{Mod}$ and $S\mathbf{Mod}$ given by extension and restriction of scalars.
\end{example}
\begin{example}\textbf{A multisorted example.}
Consider the two-sorted language ${\Sigma}$ and the theory $\mathbb{T}$ that describes sets with an action of a monoid over it. The proof of Proposition \ref{prop:ed_mod} was done in the single sorted case, but holds also in the multisorted setting. Then $\ED(\efbf^{\Sigma}_\mathbb{T},\pws )\cong \mathop{\mathrm{Mon}}\Set$; then extend the language and the theory to describe sets with an action of a group over it, so $\ED(\efbf^{\Sigma'}_\mathbb{T'},\pws )\cong \mathop{\mathrm{Grp}}\Set$. 

We can again recover the left adjoint to the forgetful functor: for a given $(M,X)$, where $M$ is a monoid acting on a set $X$, let $\mathcal{F}(M)$ be the free group generated by $M$. Define the equivalence relation $\sim$ on the product $\mathcal{F}(M)\times X$ generated by $(mn,x)\sim(m,n\cdot x)$ for any $m\in\mathcal{F}(M)$ and $n\in M$; the action of $\mathcal{F}(M)$ on $\mathcal{F}(M)\times X/\sim$ maps $(m,[(m',x)])$ into $[(mm',x)]$ for any $m,m'\in\mathcal{F}(M)$ and $x\in X$. The universal arrow is given by $(\eta_M,\iota_X):(M,X)\longrightarrow(\mathcal{F}(M),\mathcal{F}(M)\times X/\sim)$ where $\eta_M:M\to\mathcal{F}(M)$ is the inclusion of the monoid in the free group generated by it, and $\iota_X:X\to\mathcal{F}(M)\times X/\sim$ maps $x\in X$ to $[(e,x)]$, where $e$ is the identity of $M$.
\end{example}
\begin{example}\textbf{Some quasi-algebraic examples.}
Suppose we have an algebraic language ${\Sigma}$ and a quasi-algebraic theory $\mathbb{T}$, meaning that axioms can be quasi-identities---i.e.\ formulae of the form $\big(t_1(\vec x)=s_1(\vec x)\big)\land\dots\land\big(t_k(\vec x)=s_k(\vec x)\big)\vdash \big(t(\vec x)=s(\vec x)\big)$.
In this case we can recover, for example, the left adjoint to the forgetful functor between torsion-free $R\mathbf{Mod}$ and $R\mathbf{Mod}$, between cancellative semigropus and groups, between pseudocomplemented distributive lattices and boolean algebras.\end{example}
\begin{example}\textbf{Some non-algebraic example.}
Let ${\Sigma}$ be a first-order language with a binary relation $R$ and $\mathbb{T}$ a theory such that the only axiom in $\mathbb{T}$ are reflexivity, transitivity, symmetry or antisymmetry.

We can easily recover some adjunction by adding axioms---of the kind defined above---to the theory, in the same way we did for the algebraic case: for example we can find the left adjoint to the forgetful functor from the category of sets with an equivalence relation to the category of sets with a reflexive and symmetric relation, or from the category of sets with a preorder to the category of sets with an order, and so on.

A little more work must be done to recover the adjunction between the category of posets and inf-semilattices. Recall that any inf-semilattice is a poset defining that an element is smaller than another one if their meet is the first element, so there is a forgetful functor from inf-semilattices to posets. This forgetful functor arises again from a precomposition between the doctrines of Horn formulae: take the language ${\Sigma}$ with a binary predicate symbol, and a theory $\mathbb{T}$ with axioms of reflexivity, transitivity and antisymmetry; then take the algebraic language ${\Sigma'}$ with a constant symbol $\top$ and a binary function symbol $\sqcap$, and the algebraic theory $\mathbb{T'}$ that defines inf-semilattices. The functor $E:\ctx_{{\Sigma}}\to\ctx_{{\Sigma'}}$ maps projections in projections and is extended to lists of projections; $\mathfrak{e}_{\vec{x}}:\efbf^{\Sigma}_{\mathbb T}(\vec x)\to \efbf^{\Sigma'}_{\mathbb T'}(\vec x)$ is defined recursively: the top element, equalities of variables and conjunctions are sent to themselves, while the formula $R(x_i,x_j)$ is sent to the formula $\big(\sqcap(x_i,x_j)=x_i\big)$.\end{example}
\begin{example}\textbf{An example in $\Sub$.}
Consider $\Sub:\mathbb{E}\op\to\Pos$, where the base category $\mathbb{E}$ is a Grothendieck topos. Consider the empty language and the empty theory, then
\begin{equation*}\ED(\efbf,\Sub)\cong\mathbb{E}.\end{equation*} 
Extend the language with one constant symbol, so we have an equivalence of category
\begin{equation*}\ED(\efbf^{\{c\}},\Sub)\cong\mathbb{E}_{\bullet}\end{equation*} 
where $\{c\}$ is the language with one constant symbol, and $\mathbb{E}_{\bullet}$ is the category of pointed object, meaning that its objects are pairs $(A, a:\tmn\to A)$ where $A$ is an abject of $\mathbb{E}$, and arrows are homomorphism of $\mathbb{E}$ preserving the point.

So now we have the following commutative diagram:
\[\begin{tikzcd}
	{\mathbb{E}_\bullet} & {\mathbb{E}} \\
	{\ED(\efbf^{\{c\}},\Sub)} & {\ED(\efbf,\Sub)}
	\arrow["\rotatebox{-90}{$\cong$}"{description}, draw=none, from=2-1, to=1-1]
	\arrow[ from=1-1, to=1-2, "\mathcal U"]
	\arrow["{\blank\circ(E,\mathfrak{e})}", from=2-1, to=2-2]
	\arrow["\rotatebox{-90}{$\cong$}"{description}, draw=none, from=2-2, to=1-2]
\end{tikzcd}\]
where the upper arrow is the forgetful functor the leaves out the point, and the homomorphism $(E,\mathfrak{e}):\efbf\to\efbf^{\{c\}}$ is the usual arrow that arises from the extension of the empty language to the language with one constant symbol.
The left adjoint to $\mathcal{U}$ generalizes in a Grothendieck topos the classical adjoint that adds a new element to a set.\end{example}
\begin{example}\textbf{Adding an axiom.}
Consider an elementary doctrine $P$, take an element $\varphi\in P(\tmn)$. In \cite{guff2} a construction that adds $\varphi$ as an axiom is extensively studied. In particular there exists a doctrine $P_\varphi$, and an elementary doctrine homomorphism $(\id{},P(!)\varphi\land\blank):P\to P_\varphi$ such that $\varphi$ is interpreted as the top element in the fiber $P_\varphi(\tmn)$, and this homomorphism is universal with respect to this property, in the sense that every elementary homomorphism from $P$ that interprets $\varphi$ in the top element, factors (uniquely) through $(\id{},P(!)\varphi\land\blank)$---see Corollary 6.5 of \cite{guff2} for more details---.

So now take $(H,\mathfrak h):P\to \Sub$ and suppose that $\mathfrak{h}_\tmn(\varphi)=\top$. We observe that applying the left adjoint functor to $(H,\mathfrak h)$ we obtain exactly the unique $(H,\mathfrak h'):P_\varphi\to\Sub$ defined by the universal property of $(\id{},P(!)\varphi\land\blank):P\to P_\varphi$. Indeed, since the left Kan extension of $H$ along the identity is $H$ itself, it is enough to check that $\mathfrak{l}_{A}(\alpha)=\mathfrak h'(\alpha)$ for all $\alpha\in P_\varphi(A)$, so that $\mathscr{L}(A)=(HA,\mathfrak{l}_{A\times A}(\delta_A))=(HA,\Delta_{HA})$, and $Q\mathscr{L}(A)=HA$.
Consider the 1-arrow $(H,\mathfrak h'):P_\varphi\to\Sub$ and the identity 2-arrow $(H,\mathfrak h)\to(H,\mathfrak h')(\id{},P(!)\varphi\land\blank)$ along all the 1-arrows $(K,\mathfrak k):P_\varphi\to\Sub$ and the 2-arrows $\theta:(H,\mathfrak h)\to(K,\mathfrak k)(\id{},P(!)\varphi\land\blank)$. By definition of $\mathfrak l_A$ we obtain
\begin{equation*}\mathfrak l _A(\alpha)=\bigwedge_{(K,\mathfrak k),\theta}\widehat{\theta}_A^*(\mathfrak k _A(\alpha))\leq\mathfrak h'_A(\alpha).\end{equation*}
Conversely, compute
\begin{equation*}
\mathfrak{h}'_A(\alpha)=\mathfrak{h}'_A(P(!_A)\varphi\land\alpha)=\mathfrak h_A(\alpha)\leq\widehat{\theta}_A^*(\mathfrak k _A(P(!_A)\varphi\land\alpha))=\widehat{\theta}_A^*(\mathfrak k _A(\alpha))
\end{equation*}
hence $\mathfrak h'_A(\alpha)\leq\mathfrak l _A(\alpha)$.
\end{example}

\vskip 1cm
\bibliography{Biblio}
\bibliographystyle{alpha}
\end{document}